\documentclass[oneside,american,english,11pt]{amsart}
\usepackage[T1]{fontenc}
\usepackage[utf8]{inputenc}
\usepackage{babel}
\usepackage{amstext}
\usepackage{amsthm}
\usepackage{amssymb}
\makeatletter
\numberwithin{equation}{section}
\numberwithin{figure}{section}
\theoremstyle{plain}
\newtheorem{thm}{\protect\theoremname}[section]
\theoremstyle{definition}
\newtheorem{defn}[thm]{\protect\definitionname}
\theoremstyle{plain}
\newtheorem{prop}[thm]{\protect\propositionname}
\theoremstyle{plain}
\newtheorem{lem}[thm]{\protect\lemmaname}
\theoremstyle{remark}
\newtheorem{rem}[thm]{\protect\remarkname}

\usepackage{ae,aecompl} 

\makeatother

\addto\captionsamerican{\renewcommand{\definitionname}{Definition}}
\addto\captionsamerican{\renewcommand{\lemmaname}{Lemma}}
\addto\captionsamerican{\renewcommand{\propositionname}{Proposition}}
\addto\captionsamerican{\renewcommand{\remarkname}{Remark}}
\addto\captionsamerican{\renewcommand{\theoremname}{Theorem}}
\addto\captionsenglish{\renewcommand{\definitionname}{Definition}}
\addto\captionsenglish{\renewcommand{\lemmaname}{Lemma}}
\addto\captionsenglish{\renewcommand{\propositionname}{Proposition}}
\addto\captionsenglish{\renewcommand{\remarkname}{Remark}}
\addto\captionsenglish{\renewcommand{\theoremname}{Theorem}}
\providecommand{\definitionname}{Definition}
\providecommand{\lemmaname}{Lemma}
\providecommand{\propositionname}{Proposition}
\providecommand{\remarkname}{Remark}
\providecommand{\theoremname}{Theorem}

\begin{document}
\selectlanguage{american}%
\global\long\def\epsilon{\varepsilon}%

\global\long\def\phi{\varphi}%

\global\long\def\RR{\mathbb{R}}%

\global\long\def\SS{\mathbb{S}}%

\global\long\def\oo{\mathbf{1}}%

\global\long\def\dd{\mathrm{d}}%

\global\long\def\div{\operatorname{div}}%

\global\long\def\cvx{\operatorname{Cvx}}%

\global\long\def\lc{\operatorname{LC}}%

\global\long\def\dom{\operatorname{dom}}%

\global\long\def\HH{\mathcal{H}}%

\title{\selectlanguage{english}%
Surface area measures of log-concave functions}
\author{Liran Rotem}
\address{Department of Mathematics, Technion - Israel Institute of Technology,
Israel}
\email{lrotem@technion.edu}
\thanks{The author is partially supported by ISF grant 1468/19 and BSF grant
2016050.}
\begin{abstract}
This paper's origins are in two papers: One by Colesanti and Fragalà
studying the surface area measure of a log-concave function, and one
by Cordero-Erausquin and Klartag regarding the moment measure of a
convex function. These notions are the same, and in this paper we
continue studying the same construction as well as its generalization. 

In the first half the paper we prove a first variation formula for
the integral of log-concave functions under minimal and optimal conditions.
We also explain why this result is a common generalization of two
known theorems from the above papers. 

In the second half we extend the definition of the functional surface
area measure to the $L^{p}$-setting, generalizing a classic definition
of Lutwak. In this generalized setting we prove a functional Minkowski
existence theorem for even measures. This is a partial extension of
a theorem of Cordero-Erausquin and Klartag that handled the case $p=1$
for not necessarily even measures. 
\end{abstract}

\maketitle

\section{\label{sec:introduction}Functional surface area measures}

This paper has two main parts, so it also has two introductory sections.
In this section we give the necessary background concerning surface
area measures and their functional extensions. We conclude it by stating
our first main theorem, which is proved in Sections \ref{sec:first-variation}
and \ref{sec:variation-formula}. In Section \ref{sec:Lp-intro} we
introduce the Minkowski problem and its known functional analogues,
and state our second main theorem. This second theorem is proved in
Section \ref{sec:Lp-existence}.

We begin this section by recalling the classical definition of the
surface area measure of a convex body. For us, a \emph{convex body}
is a convex and compact set $K\subseteq\RR^{n}$ with non-empty interior.
The \emph{support function} $h_{K}:\RR^{n}\to\RR$ of $K$ is defined
by 
\[
h_{K}(y)=\max_{x\in K}\left\langle x,y\right\rangle ,
\]
 where $\left\langle \cdot,\cdot\right\rangle $ denotes the standard
inner product on $\RR^{n}$. For points $x\in\RR^{n}$ we will write
$\left|x\right|=\sqrt{\left\langle x,x\right\rangle }$ for the Euclidean
norm, while for convex bodies we will write $\left|K\right|$ for
their (Lebesgue) volume. The use of the same notation for both should
not cause any confusion.

If $K,L\subseteq\RR^{n}$ are convex bodies and $t\ge0$ we write
\[
K+tL=\left\{ x+ty:\ x\in K,\ y\in L\right\} 
\]
 for the Minkowski addition. A fundamental fact in convex geometry
is that for every convex body $K$ there exists a Borel measure $S_{K}$
on the unit sphere $\SS^{n-1}=\left\{ x\in\RR^{n}:\ \left|x\right|=1\right\} $
such that 
\begin{equation}
\lim_{t\to0^{+}}\frac{\left|K+tL\right|-\left|K\right|}{t}=\int_{\SS^{n-1}}h_{L}\dd S_{K}\label{eq:body-area-measure}
\end{equation}
 for every convex body $L$. The measure $S_{K}$ is called the surface
area measure of the body $K$. For a proof of this fact, as well as
alternative equivalent definitions of $S_{K}$ and general background
in convex geometry, we refer the reader to \cite{Schneider2013}. 

Over the last two decades it became more and more apparent that important
problems in convexity and asymptotic analysis can be attacked by embedding
the class of convex bodies into appropriate classes of functions or
measures on $\RR^{n}$. The idea is to think of such analytic objects
as ``generalized convex bodies'', and study geometric constructions
such as volume, addition and support functions on these larger classes.
The motivation behind this idea is twofold. First, even if one is
ultimately only interested in convex geometry and convex bodies, working
in such a larger class can be extremely useful as it allows the use
of various analytic and probabilistic tools. Second, since we are
now working with functions and measures, the geometrically inspired
theorems we obtain can be of interest in analysis. In fact, the main
theorem we prove in Section \ref{sec:Lp-existence} can be viewed
as a theorem about the existence of generalized solutions to a certain
PDE, as we will see. 

Since this fundamental idea of ``functional convexity'' is so widespread
nowadays, it is impossible to pick a manageable list of representative
papers to cite. Instead we settle for referring the reader to Section
9.5 of \cite{Schneider2013}, to the survey \cite{Milman2008} and
to the references therein. Unfortunately these only cover slightly
older results and not the massive explosion of the field over the
last decade. 

In this paper we will study functional surface area measures. We will
work with the standard class of log-concave functions:
\begin{defn}
A function $f:\RR^{n}\to[0,\infty)$ is called log-concave if for
every $x,y\in\RR^{n}$ and every $0\le\lambda\le1$, 
\[
f\left((1-\lambda)x+\lambda y\right)\ge f(x)^{1-\lambda}f(y)^{\lambda}.
\]

For every convex body $K$, the indicator function $\oo_{K}$ is log-concave.
This gives a natural embedding of the class of convex bodies into
the class of log-concave functions. We will always assume that our
log-concave functions are upper semi-continuous, which is analogous
to assuming the body $K$ is closed. We denote the class of all upper
semi-continuous log-concave functions $f:\RR^{n}\to[0,\infty)$ by
$\lc_{n}$. Every $f\in\lc_{n}$ is of the form $f=e^{-\phi}$ for
a lower semi-continuous convex function $\phi:\RR^{n}\to(-\infty,\infty]$.
We denote the class of such convex functions by $\cvx_{n}$. 

In order to have a functional analogue of \eqref{eq:body-area-measure}
we first need to understand a few things: what is the ``volume''
of log-concave functions, how to add such functions, what is the support
function $h_{f}$ of a log-concave function $f$, and most importantly
what is the surface area measure $S_{f}$ of $f$.

The easiest of the four is the volume. Since $\left|K\right|=\int\oo_{K}$,
it makes sense to define the ``volume'' of $f$ to be its Lebesgue
integral $\int f$ (unless explicitly stated otherwise, our integrals
will always be on $\RR^{n}$ with respect to the Lebesgue measure).
As for addition and support functions, recall that for convex bodies
the two operations are intimately connected by the relation $h_{K+tL}=h_{K}+th_{L}$.
If one wants to keep this relation for log-concave functions, and
have other natural properties such as monotonicity, it turns out that
there is essentially only one possible definition:
\end{defn}

\begin{thm}[\cite{Rotem2013}]
Assume we are given a map $\mathcal{T}:\lc_{n}\to\cvx_{n}$ and a
map $\oplus:\cvx_{n}\times\cvx_{n}\to\cvx_{n}$ such that 
\begin{enumerate}
\item $f\le g$ if and only if $\mathcal{T}f\le\mathcal{\mathcal{T}}g$.
\item $\mathcal{T}\oo_{K}=h_{K}$.
\item $\mathcal{T}(f\oplus g)=\mathcal{T}f+\mathcal{T}g$. 
\end{enumerate}
Then:
\begin{enumerate}
\item There exists $C>0$ such that 
\[
\left(\mathcal{T}f\right)(x)=\frac{1}{C}\cdot(-\log f)^{\ast}\left(Cx\right).
\]
\item We have 
\[
\left(f\oplus g\right)(x)=\sup_{y\in\RR^{n}}f(y)g(x-y).
\]
\end{enumerate}
\end{thm}

The $^{\ast}$ that appears in the theorem is the \emph{Legendre transform}
map $^{\ast}:\cvx_{n}\to\cvx_{n}$, defined by 
\begin{equation}
\phi^{\ast}(y)=\sup_{x\in\RR^{n}}\left(\left\langle x,y\right\rangle -\phi(x)\right).\label{eq:legendre-def}
\end{equation}
 Therefore, for a log-concave function $f=e^{-\phi}$ we define its
support function $h_{f}:\RR^{n}\to(-\infty,\infty]$ by $h_{f}=\phi^{\ast}$.
We also define the addition of log-concave functions to be 
\[
\left(f\star g\right)(x)=\sup_{y\in\RR^{n}}f(y)g(x-y).
\]
 This addition is also known as the sup-convolution, or Asplund sum,
and as an addition on log-concave functions was it first considered
in \cite{Klartag2005}. Like the support function $h_{f}$, it is
by now a standard definition in convex geometry. For $t>0$ and $f\in\lc_{n}$
we also define $\left(t\cdot f\right)(x)=f\left(\frac{x}{t}\right)^{t}$.
This dilation operation is consistent with the sup-convolution is
the sense that $2\cdot f=f\star f$. 

With these standard definitions in our disposal, we may now define:
\begin{defn}
For $f,g\in\lc_{n}$ we define 
\[
\delta(f,g)=\lim_{t\to0^{+}}\frac{\int\left(f\star\left(t\cdot g\right)\right)-\int f}{t}.
\]
 
\end{defn}

In other words, $\delta(f,g)$ is the directional derivative of the
integral at the point $f$ and in the direction $g.$ The first to
seriously study this quantity were Colesanti and Fragalà, who proved
in \cite{Colesanti2013} that $\delta(f,g)$ is well defined whenever
$\int f>0$. More importantly, they were able to prove a functional
version of \eqref{eq:body-area-measure}:
\begin{defn}
\label{def:func-surface-area}For $f=e^{-\phi}\in\lc_{n}$, we define
its surface area measure $S_{f}$ as the push-forward of the measure
$f\dd x$ under the map $\nabla\phi$. More explicitly, for every
Borel subset $A\subseteq\RR^{n}$ we define 
\[
S_{f}(A)=\int_{\left\{ x:\ \nabla\phi(x)\in A\right\} }f\dd x.
\]
\end{defn}

Equivalently, if $f=e^{-\phi}$ then $S_{f}$ is the unique Borel
measure on $\RR^{n}$ such that 
\begin{equation}
\int_{\RR^{n}}\rho(y)\dd S_{f}(y)=\int_{\RR^{n}}\rho\left(\nabla\phi(x)\right)f(x)\dd x\label{eq:push-forward}
\end{equation}
 for all functions $\rho$ for which the left hand side is well defined
(it could be $\pm\infty$). Note that this definition does not require
any regularity assumptions from $f$. Indeed, since $\phi$ is a convex
function it is differentiable almost everywhere on the set $\dom\phi=\left\{ x\in\RR^{n}:\ \phi(x)<\infty\right\} $
\textendash{} see e.g. \cite{Rockafellar1970} for this fact as well
as other basic analytic properties of convex functions. Therefore
$\nabla\phi$ exists $f\dd x$-a.e. , which means the push-forward
is well defined. 

Definition \ref{def:func-surface-area} can seem a bit strange at
first. For example, we note that $S_{f}\left(\RR^{n}\right)=\int f$,
which is the ``volume'' of $f$ and not its ``surface area''.
This is unlike the classical case of convex bodies, where we have
$S_{K}\left(\SS^{n-1}\right)=\left|\partial K\right|$, the surface
area of $K$. However, at least for sufficiently regular functions
$f$ it turns out that $S_{f}$ is indeed the correct definition for
a surface area measure, because of the following theorem of Colesanti
and Fragalà:
\begin{thm}[\cite{Colesanti2013}]
\label{thm:colesanti-fragala}Fix $f,g\in\lc_{n}$. Assume that $\phi=-\log f$
and $\psi=-\log g$ belong to the class 
\[
\left\{ \rho\in\cvx_{n}:\ \begin{array}{l}
\rho\text{ is finite and }C^{2}\text{-smooth, }\text{\ensuremath{\nabla^{2}\rho(x)\succ0}}\\
\text{for all }x\in\RR^{n},\text{ and }\text{\ensuremath{\lim_{\left|x\right|\to\infty}\frac{\rho(x)}{\left|x\right|}=\infty}}
\end{array}\right\} .
\]
 Assume further that $h_{f}-ch_{g}$ is a convex function for sufficiently
small $c>0$. Then 
\begin{equation}
\delta(f,g)=\int_{\RR^{n}}h_{g}\dd S_{f}.\label{eq:func-first-var}
\end{equation}
 
\end{thm}

Theorem \ref{thm:colesanti-fragala} was one of the main motivations
behind this paper. Comparing it with \eqref{eq:body-area-measure}
explains why $S_{f}$ should indeed be considered as the surface area
measure of the function $f$. At the same time, the assumptions of
the theorem are definitely not optimal. For example, it was already
proved in \cite{Rotem2012} that when $f(x)=e^{-\left|x\right|^{2}/2}$,
the first variation formula \eqref{eq:func-first-var} holds for all
$g\in\lc_{n}$ with no technical assumptions. 

However, there is no doubt that some assumptions are needed, as \eqref{eq:func-first-var}
cannot hold for all functions $f,g\in\lc_{n}$. For example, if $f=\oo_{K}$
for some convex body $K$ then $S_{f}=\left|K\right|\cdot\delta_{0}$,
so \eqref{eq:func-first-var} cannot hold even if $g$ is also the
indicator of a convex body. We therefore see that the theory developed
in this paper is a functional analogue of the classical theory for
convex bodies, but does not formally extend it. We remark that in
\cite{Colesanti2013} the authors did define a generalization of the
measure $S_{f}$ to the case where $f$ is supported on some smooth
convex body $K$ and satisfies some technical assumptions, and proved
a generalization of Theorem \ref{thm:colesanti-fragala} for such
functions $f$. Unfortunately the indicator function $\oo_{K}$ does
not satisfy these technical assumptions, so even this more general
theorem does not recover the case of convex bodies. It's an interesting
open problem to find extensions of Definition \ref{def:func-surface-area}
and Theorem \ref{thm:colesanti-fragala} that will hold for all $f\in\lc_{n}$,
including both smooth functions and indicators of convex bodies. In
this paper, however, we will keep Definition \ref{def:func-surface-area}
as our definition of $S_{f}$, and prove a necessary and sufficient
condition on $f$ for \eqref{eq:func-first-var} to hold. 

The correct condition to impose on $f$ appeared in the work of Cordero-Erausquin
and Klartag (\cite{Cordero-Erausquin2015}). In this work the authors
study the moment measure of a convex function $\phi$. While this
term comes from the very different field of toric Kähler manifolds,
the moment measure of $\phi$ is precisely the same measure as the
surface area measure $S_{e^{-\phi}}$. We refer the reader to \cite{Cordero-Erausquin2015}
and the previous works cited therein (especially \cite{Berman2013})
for more information about the connection between this measure and
complex geometry. The results of \cite{Cordero-Erausquin2015} and
of Section \ref{sec:Lp-existence} below can also be considered as
establishing the existence of generalized solutions to certain Monge\textendash Ampère
differential equations. This shows again how functional results that
are motivated by convex geometry can have applications to very different
areas of mathematics. 

In any case, the crucial definition from \cite{Cordero-Erausquin2015}
is the following:
\begin{defn}
Fix $f\in\lc_{n}$ with $0<\int f<\infty$. We say that $f$ is essentially
continuous if 
\[
\HH^{n-1}\left(\left\{ x\in\RR^{n}:\ f\text{ is not continuous at }x\right\} \right)=0,
\]
 where $\HH^{n-1}$ denotes the $(n-1$)-dimensional Hausdorff measure
(see, e.g. Chapter 2 of \cite{Evans1992} for definition of the Hausdorff
measure).
\end{defn}

To explain this definition, let us write $K=\overline{\left\{ x:\ f(x)\ne0\right\} }$
for the support of $f$. Since $f=e^{-\phi}$ for a convex function
$\phi$ and convex functions are continuous on the interior of their
domain, $f$ is continuous everywhere outside of $\partial K$. Moreover,
since $f$ is upper semi-continuous, it is easy to see that $f$ is
continuous at a boundary point $x_{0}\in K$ if and only if $f(x_{0})=0$.
Therefore $f$ is essentially continuous if and only if $f\equiv0$
$\HH^{n-1}$-a.e. on $\partial K$. 

The following theorem is an adaptation of Theorem 8 from \cite{Cordero-Erausquin2015}
to our notation:
\begin{thm}[\cite{Cordero-Erausquin2015}]
\label{thm:subdifferential}Fix $f,g\in\lc_{n}$ with $0<\int f,\int g<\infty$.
Assume that $f$ is essentially continuous. Then 
\[
\log\int f-\log\int g\ge\frac{1}{\int f}\cdot\int\left(h_{f}-h_{g}\right)\dd S_{f}.
\]
\end{thm}

The relationship between Theorems \ref{thm:colesanti-fragala} and
\ref{thm:subdifferential} may not be immediately clear, and as far
as we know did not previously appear in the literature. To understand
it, let us define $F:\cvx_{n}\to[-\infty,\infty]$ by 
\[
F(\psi)=-\log\int e^{-\psi^{\ast}}.
\]
 In other words for every $f\in\lc_{n}$ we have $F(h_{f})=-\log\int f$,
where we are using the fact that $\psi^{\ast\ast}=\psi$ for all $\psi\in\cvx_{n}$.
Theorem \ref{thm:colesanti-fragala} is a theorem about the differential
of $F$ at some point $h_{f}$. Indeed, under its technical assumptions
we have by the chain rule 
\begin{align*}
\left.\frac{\dd}{\dd t}\right|_{t=0^{+}}F\left(h_{f}+th_{g}\right) & =\left.\frac{\dd}{\dd t}\right|_{t=0^{+}}F\left(h_{f\star\left(t\cdot g\right)}\right)=-\left.\frac{\dd}{\dd t}\right|_{t=0^{+}}\left[\log\int\left(f\star\left(t\cdot g\right)\right)\right]\\
 & =-\frac{1}{\int f}\cdot\left.\frac{\dd}{\dd t}\right|_{t=0^{+}}\int\left(f\star\left(t\cdot g\right)\right)=-\frac{1}{\int f}\cdot\int h_{g}\dd S_{f}.
\end{align*}
 In other words, the linear map $L_{f}(h_{g})=-\frac{1}{\int f}\int h_{g}\dd S_{f}$
is the differential of $F$ at the point $h_{f}$. On the other hand,
Theorem \ref{thm:subdifferential} is a theorem about the \emph{subdifferential}
of $F$. Indeed, the Prékopa\textendash Leindler inequality (\cite{Prekopa1971},
\cite{Leindler1972}) states that for every $f,g\in\lc_{n}$ with
$0<\int f,\int g<\infty$ and $0<t<1$ one has 
\[
\int\left(\left(1-t\right)\cdot f\right)\star\left(t\cdot g\right)\ge\left(\int f\right)^{1-t}\left(\int g\right)^{t},
\]
 or 
\[
F\left((1-t)h_{f}+th_{g}\right)\le(1-t)F(h_{f})+tF(h_{g}).
\]
 Therefore $F$ is convex on the appropriate domain. The conclusion
of Theorem \ref{thm:subdifferential} may be written as 
\[
F(h_{g})-F(h_{f})\ge L_{f}\left(h_{g}-h_{f}\right),
\]
 which means that $L_{f}$ belongs to the \emph{subdifferential} $\partial F(h_{f})$.
This conclusion is weaker than the conclusion of Theorem \ref{thm:colesanti-fragala},
but the assumptions are much weaker as well. In fact, Cordero-Erausquin
and Klartag show that essential continuity of $f$ is necessary for
Theorem \ref{thm:subdifferential} to hold. 

The first major goal of this paper is to prove a common generalization
of both Theorems \textendash{} We will obtain the stronger conclusion
of Theorem \ref{thm:colesanti-fragala} under the optimal assumptions
of Theorem \ref{thm:subdifferential}. More concretely, we will prove
the following result:
\begin{thm}
\label{thm:first-variation}Fix $f,g\in\lc_{n}$ with $0<\int f,\int g<\infty$.
If $f$ is essentially continuous, then
\begin{equation}
\delta(f,g)=\int h_{g}\dd S_{f}.\label{eq:moment-first-variation-1}
\end{equation}
 Moreover, \eqref{eq:moment-first-variation-1} holding for $g=\oo_{B_{2}^{n}}$
is \textbf{equivalent} to $f$ being essentially continuous. 
\end{thm}

We immediately remark that it is not necessarily true that both sides
of \eqref{eq:moment-first-variation-1} are finite, and this equality
can take the form $+\infty=+\infty$. As an example it is enough to
take $f(x)=e^{-\left|x\right|^{2}/2}$, so $S_{f}=f\dd x$, and the
function $g\in\lc_{n}$ which satisfies $h_{g}(x)=e^{\left|x\right|^{2}}$. 

Theorem \ref{thm:first-variation} is proved in Section \ref{sec:variation-formula},
after some preliminary technicalities are proved in Section \ref{sec:first-variation}.
Other than a general desire to state theorems under the minimal and
most elegant conditions, we believe that our proof also explains in
a very transparent way exactly \emph{why} essential continuity is
the natural condition here. The proof also hints about possible extensions
of the theorem to the non essentially continuous case.

\section{\label{sec:first-variation}First variation of the Legendre transform}

This section is fairly short and technical, and is dedicated to a
proof of the following result:
\begin{prop}
\label{prop:pointwise-der}Let $\psi,\alpha:\RR^{n}\to(-\infty,\infty]$
be lower semi-continuous functions. Assume that $\alpha$ is bounded
from below and that $\alpha(0),\psi(0)<\infty$. Write $\phi=\psi^{\ast}$
. Then at every point $x_{0}\in\RR^{n}$ where $\phi$ is differentiable
we have 

\begin{equation}
\left.\frac{\dd}{\dd t}\right|_{t=0^{+}}\left(\psi+t\alpha\right)^{\ast}(x_{0})=-\alpha\left(\nabla\phi(x_{0})\right).\label{eq:legendre-first-variation}
\end{equation}
\end{prop}

Conceptually, Proposition \ref{prop:pointwise-der} is well-known.
For example, it is similar to Lemma 4.11 of \cite{Colesanti2013}.
In fact, if $\left\{ \psi_{t}\right\} $ is any family of convex functions,
then it is well-known that under sufficient regularity assumptions
we have 
\begin{equation}
\left.\frac{\dd}{\dd t}\right|_{t=0}\psi_{t}^{\ast}(x_{0})=-\left.\frac{\dd}{\dd t}\right|_{t=0}\psi_{t}\left(\nabla\psi_{0}^{\ast}(x_{0})\right)\label{eq:legendre-variation-general}
\end{equation}
 (This result is folklore, but see e.g. Proposition 5.1 of \cite{Artstein-Avidan2017}
for one rigorous formulation). The thorny issue here is the words
``sufficient regularity assumptions'': for the proof of Theorem
\ref{thm:first-variation} we will need Proposition \ref{prop:pointwise-der}
as stated, with no extra smoothness or boundness assumptions. In fact,
we do not even assume that $\psi$ and $\alpha$ are convex \textendash{}
The Legendre transform of any function $\phi:\RR^{n}\to(-\infty,\infty]$,
convex or not, can be defined using formula \eqref{eq:legendre-def}.
This will not be important for the proof of Theorem \ref{thm:first-variation},
but will be useful in the second half of this paper.

As we were unable to find Proposition \ref{prop:pointwise-der} in
the literature with our minimal assumptions, we give a full proof
in this section. We do mention that Lemma 2.7 of \cite{Berman2013}
is fairly close to our proposition, and the proofs will have some
similarities as well. We begin with a lemma:
\begin{lem}
\label{lem:unique-maximizer}Let $\psi:\RR^{n}\to(-\infty,\infty]$
be a lower semi-continuous function and let $\phi=\psi^{\ast}$ .
Assume that for some fixed $x_{0}\in\RR^{n}$ the function $\phi$
is differentiable at $x_{0}$. Then:
\begin{enumerate}
\item $\lim_{\left|y\right|\to\infty}\left(\left\langle y,x_{0}\right\rangle -\psi(y)\right)=-\infty$.
\item The supremum in the definition of $\psi^{\ast}(x_{0})=\phi(x_{0})$
is attained at the unique point $y_{0}=\nabla\phi(x_{0})$.
\end{enumerate}
\end{lem}

\begin{proof}
For $(1)$ we do not need that $\phi$ is differentiable at $x_{0}$,
but only that it is finite in a neighborhood of $x_{0}$. Since $\phi$
is convex, it follows that there exists $\epsilon>0$ and $M>0$ such
that $\phi\le M$ on $\overline{B}(x_{0},\epsilon)$. 

For any $0\ne y\in\RR^{n}$ we have 
\[
M\ge\phi\left(x_{0}+\epsilon\frac{y}{\left|y\right|}\right)\ge\left\langle y,x_{0}+\epsilon\frac{y}{\left|y\right|}\right\rangle -\psi(y)=\left(\left\langle y,x_{0}\right\rangle -\psi(y)\right)+\epsilon\left|y\right|.
\]
 Hence 
\[
\left\langle y,x_{0}\right\rangle -\psi(y)\le M-\epsilon\left|y\right|\xrightarrow{\left|y\right|\to\infty}-\infty,
\]
 so $(1)$ is proved. 

Now we prove $(2)$. Since the function 
\[
y\mapsto\left\langle y,x_{0}\right\rangle -\psi(y)
\]
 is upper semi-continuous and tends to $-\infty$ as $\left|y\right|\to\infty$,
it must attain a maximum at some point $y_{0}$. We will show that
necessarily $y_{0}=\nabla\phi(x_{0})$, which will also imply that
the maximizer $y_{0}$ is unique. 

Indeed, for every $v\in\RR^{n}$ and every small $t>0$ we have 
\begin{align*}
\phi(x_{0}+tv) & \ge\left\langle y_{0},x_{0}+tv\right\rangle -\psi(y_{0})=\left\langle y_{0},x_{0}+tv\right\rangle -\left(\left\langle y_{0},x_{0}\right\rangle -\phi(x_{0})\right)\\
 & =\phi(x_{0})+t\left\langle y_{0},v\right\rangle .
\end{align*}
 Hence 
\[
\left\langle y_{0},v\right\rangle \le\frac{\phi(x_{0}+tv)-\phi(x_{0})}{t}\xrightarrow{t\to0^{+}}\left\langle \nabla\phi(x_{0}),v\right\rangle .
\]
 By replacing $v$ with $-v$ we have $\left\langle y_{0},v\right\rangle =\left\langle \nabla\phi(x_{0}),v\right\rangle $
for all $v\in\RR^{n}$, so indeed $y_{0}=\nabla\phi(x_{0})$ and $(2)$
is proved.
\end{proof}
We can now prove Proposition \ref{prop:pointwise-der}:
\begin{proof}[Proof of Proposition \ref{prop:pointwise-der}]
Choose $M>0$ such that $\alpha\ge-M$. 

By definition we have 
\[
\phi_{t}(x_{0})=\sup_{y\in\RR^{n}}\left(\left\langle x_{0},y\right\rangle -\psi_{t}(y)\right)=\sup_{y\in\RR^{n}}\underbrace{\left(\left\langle x_{0},y\right\rangle -\psi(y)-t\alpha(y)\right)}_{=:G_{t}(y)}.
\]

According to Lemma \ref{lem:unique-maximizer}$(1)$ we know that
$\lim_{\left|y\right|\to\infty}G_{0}(y)=-\infty.$ Since $G_{t}\le G_{0}+tM$
we also have $\lim_{\left|y\right|\to\infty}G_{t}(y)=-\infty$. Since
$\psi$ and $\alpha$ are lower semi-continuous it follows that the
$\sup_{y}G_{t}(y)$ is attained at some point $y_{t}$. We claim that
the set $\left\{ y_{t}\right\} _{0\le t\le1}$ is bounded. Indeed,
for $0\le t\le1$ we have
\begin{align*}
G_{0}(y_{t}) & \ge G_{t}(y_{t})-tM=\sup_{y\in\RR^{n}}G_{t}(y)-tM\\
 & \ge\sup_{y\in\RR^{n}}\left(\left\langle x_{0},y\right\rangle -\psi(y)-\alpha(y)-M\right)\ge-\psi(0)-\alpha(0)-M>-\infty,
\end{align*}
Since $\lim_{\left|y\right|\to\infty}G_{0}(y)=-\infty,$ $\left\{ y_{t}\right\} _{0\le t\le1}$
is indeed bounded. For $t=0$ we know from Lemma \ref{lem:unique-maximizer}$(2)$
that $G_{0}(y)$ is maximized at the \emph{unique }point $y_{0}=\nabla\phi_{0}(x_{0})$.

On the one hand, we have 
\[
\phi_{t}(x_{0})\ge G_{t}(y_{0})=G_{0}(y_{0})-t\alpha(y_{0})=\phi_{0}(x_{0})-t\alpha(y_{0}),
\]
so we have the bound
\[
\liminf_{t\to0^{+}}\frac{\phi_{t}(x_{0})-\phi_{0}(x_{0})}{t}\ge-\alpha(y_{0}).
\]

On the other hand, we have 
\[
\phi_{0}(x_{0})\ge G_{0}(y_{t})=G_{t}(y_{t})+t\alpha(y_{t})=\phi_{t}(x_{0})+t\alpha(y_{t})
\]
 so 
\[
\limsup_{t\to0^{+}}\frac{\phi_{t}(x_{0})-\phi_{0}(x_{0})}{t}\le-\liminf_{t\to0^{+}}\alpha(y_{t}).
\]
 Therefore, to finish the proof it is enough to show that $\liminf_{t\to0^{+}}\alpha(y_{t})\ge\alpha(y_{0})$.
Since $\alpha$ is lower semi-continuous, it is enough to show that
$y_{t}\xrightarrow{t\to0^{+}}y_{0}$. Assume by contradiction this
is not the case. Since $\left\{ y_{t}\right\} _{0\le t\le1}$ is bounded
we can find a converging sequence $t_{i}\to0$ such that $y_{t_{i}}\to y^{\ast}\ne y$.
Since $G_{0}$ is upper semi-continuous it follows that $G_{0}(y^{\ast})\ge\limsup_{i\to\infty}G_{0}(y_{t_{i}})$.
But since $y_{t_{i}}$ maximizes $G_{t_{i}}$ we have
\[
G_{0}(y_{t_{i}})+t_{i}M\ge G_{t_{i}}(y_{t_{i}})\ge G_{t_{i}}(y_{0})=G_{0}(y_{0})-t_{i}\alpha(y_{0}).
\]
 Since $t_{i}\to0$ as $i\to\infty$ it follows that
\[
G_{0}\left(y^{\ast}\right)\ge\limsup_{i\to\infty}G_{0}(y_{t_{i}})\ge G_{0}(y_{0}),
\]
 contradicting the fact that $y_{0}$ is the \emph{unique} maximizer
of $G_{0}$. Hence $y_{t}\to y_{0}$ and \eqref{eq:legendre-first-variation}
is proved. 
\end{proof}

\section{\label{sec:variation-formula}Essentially continuity and the variation
Formula}

We now begin our proof of Theorem \ref{thm:first-variation}. The
following fact about essentially continuous log-concave functions
was proved in \cite{Cordero-Erausquin2015}:
\begin{prop}
\label{prop:int-grad}For every $f\in\lc_{n}$ with $0<\int f<\infty$
we have $\int\left|\nabla f\right|<\infty$. If $f$ is essentially
continuous then $\int\nabla f=0$. 
\end{prop}

The following result is the main place essential continuity is used
in our proof. It may be of independent interest:
\begin{thm}
\label{thm:coarea}Fix $f\in\lc_{n}$ with $0<\int f<\infty$ and
let $K=\overline{\left\{ x:\ f(x)\ne0\right\} }$ denote its support.
Then 
\[
\int_{0}^{\infty}\mathcal{H}^{n-1}\left(\left\{ x:\ f(x)=t\right\} \right)\dd t=\int_{\RR^{n}}\left|\nabla f\right|\dd x+\int_{\partial K}f\dd\HH^{n-1}.
\]
 In particular, $f$ is essentially continuous if and only if we have
the classic co-area formula 
\begin{equation}
\int_{0}^{\infty}\mathcal{H}^{n-1}\left(\left\{ x:\ f(x)=t\right\} \right)\dd t=\int_{\RR^{n}}\left|\nabla f\right|\dd x.\label{eq:lc-coarea}
\end{equation}
 
\end{thm}

\begin{proof}
We will use the co-area formula for BV functions \textendash{} see
e.g. \cite{Evans1992} for the statement and the necessary definitions.
By translating $f$ we may assume without loss of generality that
$0$ is in the interior of $K$. Let $\Phi:\RR^{n}\to\RR^{n}$ be
a $C^{1}$ vector field with compact support. Choose a ball $B$ such
that $\textrm{support}(\Phi)\subseteq B$. For every $\lambda<1$
the set $\lambda K\cap B$ is convex, hence a Lipschitz domain. Since
convex functions are locally Lipschitz on the interior of their support,
it follows that $f$ is Lipschitz on $\lambda K\cap B$. Hence $f\Phi$
is also Lipschitz and we may apply the divergence theorem:
\begin{align*}
\int_{\partial(\lambda K\cap B)}\left\langle f\Phi,n_{\lambda K\cap B}\right\rangle \dd\HH^{n-1} & =\int_{\lambda K\cap B}\div\left(f\Phi\right)=\int_{\lambda K}\div\left(f\Phi\right)\\
 & =\int_{\lambda K}\left\langle \nabla f,\Phi\right\rangle +\int_{\lambda K}f\div\Phi.
\end{align*}
 Here of course $n_{\lambda K\cap B}$ denotes the outer unit normal
to the set $\lambda K\cap B$, which exists $\HH^{n-1}$-almost everywhere. 

Since $\Phi\equiv0$ on $\partial B$ we also have 
\begin{align*}
\int_{\partial(\lambda K\cap B)}\left\langle f\Phi,n_{\lambda K\cap B}\right\rangle \dd\HH^{n-1} & =\int_{\partial\left(\lambda K\right)}f(y)\left\langle \Phi(y),n_{\lambda K}(y)\right\rangle \dd\HH^{n-1}(y)\\
 & =\lambda^{n-1}\int_{\partial K}f(\lambda x)\left\langle \Phi(\lambda x),n_{\lambda K}(\lambda x)\right\rangle \dd\HH^{n-1}(x)\\
 & =\lambda^{n-1}\int_{\partial K}f(\lambda x)\left\langle \Phi(\lambda x),n_{K}(x)\right\rangle \dd\HH^{n-1}(x).
\end{align*}
 Letting $\lambda\to1^{-}$ and using the dominated convergence theorem
we obtain 
\[
\int_{\partial K}f\left\langle \Phi,n_{K}\right\rangle \dd\HH^{n-1}=\int_{K}\left\langle \nabla f,\Phi\right\rangle +\int_{K}f\div\Phi,
\]
 and since $f$ is supported on $K$ we may also write 
\[
\int_{\RR^{n}}f\div\Phi=-\int_{\RR^{n}}\left\langle \nabla f,\Phi\right\rangle +\int_{\partial K}f\left\langle \Phi,n_{K}\right\rangle \dd\HH^{n-1}.
\]
 By definition, this means that $f$ is a function of locally bounded
variation, and its variation measure $\left\Vert Df\right\Vert $
satisfies
\begin{equation}
\dd\left\Vert Df\right\Vert =\left|\nabla f\right|\dd x+f\cdot\left.\dd\HH^{n-1}\right|_{\partial K}\label{eq:variation-measure}
\end{equation}
 (see Section 5.1 of \cite{Evans1992}). In particular, we may apply
the co-area formula (Section 5.5 of \cite{Evans1992}) and conclude
that 
\[
\int_{0}^{\infty}\mathcal{H}^{n-1}\left(\left\{ x:\ f(x)=t\right\} \right)\dd t=\left\Vert Df\right\Vert \left(\RR^{n}\right)=\int\left|\nabla f\right|\dd x+\int_{\partial K}f\dd\HH^{n-1},
\]

which is what we wanted to prove. 

Finally, for the ``in particular'' part of the theorem, we see from
the last equation that \eqref{eq:lc-coarea} holds if and only if
$\int_{\partial K}f\dd\HH^{n-1}=0$. This holds if and only if $f\equiv0$
$\HH^{n-1}$-a.e. on $\partial K$, which exactly means that $f$
is essentially continuous. 
\end{proof}
\begin{rem}
Equation \eqref{eq:variation-measure} actually shows that $f$ is
essentially continuous if and only if its variation measure $\left\Vert Df\right\Vert $
is absolutely continuous with respect to the Lebesgue measure. This
is equivalent to $f$ belonging to the Sobolev space $W_{loc}^{1,1}\left(\RR^{n}\right)$
\textendash{} see again Section 5.1 of \cite{Evans1992}. By \eqref{eq:variation-measure}
and Proposition \ref{prop:int-grad} we know that in this case
\[
\left\Vert Df\right\Vert \left(\RR^{n}\right)=\int\left|\nabla f\right|<\infty,
\]
 so we actually obtain the following characterization: a function
$f\in\lc_{n}$ with $0<\int f<\infty$ is essentially continuous if
and only if $f\in W^{1,1}\left(\RR^{n}\right)$. We will not need
this characterization in this paper. 
\end{rem}

We can already prove the ``moreover'' part of Theorem \ref{thm:first-variation}.
In fact we will show something slightly more general:
\begin{prop}
\label{prop:first-variation-ball}Fix $f\in\lc_{n}$ with $0<\int f<\infty$.
Write $g=\lambda\cdot\oo_{mB_{2}^{n}}$ for some $\lambda,m>0$. Then
\begin{equation}
\delta(f,g)=\int h_{g}\dd S_{f}\label{eq:variation-ball}
\end{equation}
 if and only if $f$ is essentially continuous. 
\end{prop}

\begin{proof}
We first observe that multiplying $g$ by a constant cannot change
the validity of \eqref{eq:variation-ball}. Indeed, define $\widetilde{g}=e^{c}\cdot g$
for some $c\in\RR$. Then on the left hand side we obtain
\begin{align*}
\delta(f,\widetilde{g}) & =\left.\frac{\dd}{\dd t}\right|_{t=0^{+}}\int\left(f\star\left(t\cdot\widetilde{g}\right)\right)=\left.\frac{\dd}{\dd t}\right|_{t=0^{+}}\left[e^{tc}\cdot\int\left(f\star\left(t\cdot g\right)\right)\right]\\
 & =c\cdot\int f+\left.\frac{\dd}{\dd t}\right|_{t=0^{+}}\int\left(f\star\left(t\cdot g\right)\right)=c\int f+\delta(f,g),
\end{align*}
 While on the right hand side we obtain 
\[
\int h_{\widetilde{g}}\dd S_{f}=\int\left(h_{g}+c\right)\dd S_{f}=c\int\dd S_{f}+\int h_{g}\dd S_{f}=c\int f+\int h_{g}\dd S_{f}.
\]
 Since both sides changed by the same additive term, the validity
of \eqref{eq:variation-ball} did not change. Hence we may assume
without loss of generality that $\lambda=1$, i.e. $g=\oo_{mB_{2}^{n}}$. 

Let us compute both sides of \eqref{eq:variation-ball}. Write $f_{t}=f\star\left(t\cdot g\right)$,
so by definition
\[
f_{t}(x)=\sup_{y\in\RR^{n}}f(x-y)\oo_{mB_{2}^{n}}\left(\frac{y}{t}\right)^{t}=\sup_{y\in tmB_{2}^{n}}f(x-y).
\]
It follows that if we set $K_{s}=\left\{ x:\ f(x)\ge s\right\} $
then 
\[
\left\{ x:\ f_{t}(x)\ge s\right\} =K_{s}+tmB_{2}^{n},
\]
so by the layer cake representation we have 
\[
\frac{\int f_{t}-\int f}{t}=\frac{\int_{0}^{\infty}\left|K_{s}+tB_{2}^{n}\right|\dd s-\int_{0}^{\infty}\left|K_{s}\right|\dd s}{t}=\int_{0}^{\infty}\frac{\left|K_{s}+tmB_{2}^{n}\right|-\left|K_{s}\right|}{t}\dd s.
\]
For every $s>0$ we have 
\[
\frac{\left|K_{s}+tmB_{2}^{n}\right|-\left|K_{s}\right|}{t}=m\frac{\left|K_{s}+tmB_{2}^{n}\right|-\left|K_{s}\right|}{tm}\xrightarrow{t\to0^{+}}m\cdot\left|\partial K_{s}\right|.
\]
 Moreover, Minkowski's polynomiality theorem (see e.g. Theorem 5.1.7
of \cite{Schneider2013}) implies that for every fixed $s>0$ the
left hand side is a polynomial in $t$ with non-negative coefficients,
and hence monotone in $t$. Therefore we may apply the monotone convergence
theorem and deduce that 
\[
\lim_{t\to0^{+}}\frac{\int f_{t}-\int f}{t}=m\int_{0}^{\infty}\left|\partial K_{s}\right|\dd s.
\]

On the other hand, we have $h_{g}(y)=m\left|y\right|$, so 
\[
\int h_{g}\dd S_{f}=m\int\left|\nabla\left(-\log f\right)\right|f=m\int\left|\nabla f\right|.
\]

Therefore in the case $g=\oo_{mB_{2}^{n}}$, formula \eqref{eq:variation-ball}
reduces to the co-area formula \eqref{eq:lc-coarea}. By Theorem \ref{thm:coarea},
it holds if and only if $f$ is essentially continuous. 
\end{proof}
Proposition \ref{prop:first-variation-ball} explains the role of
essential continuity in the subject, but the full proof of Theorem
\ref{thm:first-variation} is more technically involved. We will need
Proposition \ref{prop:pointwise-der} from Section \ref{sec:first-variation},
as well as two more results. The first is contained e.g. in Lemma
3.2 of \cite{Artstein-Avidan2004}:
\begin{prop}
\label{prop:lc-int-conv}Let $f,f_{1},f_{2},\ldots:\RR^{n}\to[0,\infty)$
be log-concave functions such that $f_{i}\xrightarrow{i\to\infty}f$
pointwise. Then $\int f_{i}\to\int f$. 
\end{prop}

The second is a very simple measure theoretic lemma:
\begin{lem}
\label{lem:fatou-comparison}Let $\left\{ u_{t}\right\} _{t>0},\left\{ v_{t}\right\} _{t>0},\left\{ w_{t}\right\} _{t>0}$
be families of integrable functions $u_{t},v_{t},w_{t}:\RR^{n}\to\RR$
such that $u_{t}\xrightarrow{t\to0^{+}}u$ , $v_{t}\xrightarrow{t\to0^{+}}v$
and $w_{t}\xrightarrow{t\to0^{+}}w$ almost everywhere. Assume that:
\begin{enumerate}
\item $u_{t}\le v_{t}\le w_{t}$ for all $t>0$.
\item $\int w_{t}\xrightarrow{t\to0^{+}}\int w<\infty$.
\item $\int u_{t}\xrightarrow{t\to0^{+}}\int u>-\infty$.
\end{enumerate}
Then we also have $\int v_{t}\xrightarrow{t\to0^{+}}\int v$. 

\end{lem}

\begin{proof}
Applying Fatou's lemma to $w_{t}-v_{t}$ we have 
\[
\int w-\limsup_{t\to0^{+}}\int v_{t}=\liminf_{t\to0^{+}}\int\left(w_{t}-v_{t}\right)\ge\int\left(w-v\right)=\int w-\int v,
\]
 so $\limsup_{t\to0^{+}}\int v_{t}\le\int v$. Similarly we may apply
Fatou's lemma to $v_{t}-u_{t}$ and obtain 
\[
\liminf_{t\to0^{+}}\int v_{t}-\int u=\liminf_{t\to0^{+}}\int\left(v_{t}-u_{t}\right)\ge\int(v-u)=\int v-\int u,
\]
 so $\liminf_{t\to0^{+}}\int v_{t}\ge\int v$. The claim follows. 
\end{proof}
We are now ready to prove Theorem \ref{thm:first-variation}. We will
need a bit of notation for the proof. First, we write $f=e^{-\phi}$
and $g=e^{-\beta}$. We also set $\psi=h_{f}=\phi^{\ast}$ and $\alpha=h_{g}=\beta^{\ast}$.
Finally we define $\psi_{t}=\psi+t\alpha$, $\phi_{t}=\psi_{t}^{\ast}$
and $f_{t}=e^{-\phi_{t}}=f\star\left(t\cdot g\right)$. 

We first prove the theorem under the extra assumption $\alpha$ grows
very slowly:
\begin{lem}
\label{lem:var-linear-growth}Under the assumptions of Theorem \ref{thm:first-variation}
assume further that 
\[
-m\le h_{g}(y)\le m\left|y\right|+c
\]
 for some $m,c>0$. Then 
\[
\delta(f,g)=\int h_{g}\dd S_{f}.
\]
\end{lem}

\begin{rem}
\label{rem:finiteness-linear}Unlike the more general Theorem \ref{thm:first-variation},
the equality in the lemma is always an equality of finite quantities.
Indeed, 
\[
\int h_{g}\dd S_{f}=\int h_{g}\left(\nabla\phi\right)e^{-\phi}\le\int\left(m\left|\nabla\phi\right|+c\right)e^{-\phi}=m\int\left|\nabla\left(e^{-\phi}\right)\right|+c\int e^{-\phi},
\]
 which is finite by Proposition \ref{prop:int-grad}. 
\end{rem}

\begin{proof}
Write $\widetilde{g}=e^{c}\cdot\oo_{mB_{2}^{n}}$, and observe that
$h_{\widetilde{g}}(y)=m\left|y\right|+c$. Define 
\[
\widetilde{f}_{t}(x)=\left(f\star\left(t\cdot\widetilde{g}\right)\right)(x)=e^{tc}\cdot\max_{z:\ \left|z-x\right|\le mt}f(z).
\]
 Since $h_{g}\le h_{\widetilde{g}}$ we also have $g\le\widetilde{g}$,
so for all $t>0$ we have $f_{t}\le\widetilde{f}_{t}$. On the other
hand, we also have 
\[
f_{t}=e^{-\left(\psi+t\alpha\right)^{\ast}}\ge e^{-\left(\psi-tm\right)^{\ast}}=e^{-\left(\phi+tm\right)}=e^{-tm}f.
\]

Therefore if we define $u_{t}=\frac{e^{-tm}f-f}{t}$, $v_{t}=\frac{f_{t}-f}{t}$
and $w_{t}=\frac{\widetilde{f}_{t}-f}{t}$ then $u_{t}\le v_{t}\le w_{t}$
for all $t>0$. 

We claim that 
\begin{equation}
v_{t}\to\alpha\left(\nabla\phi\right)f\label{eq:pointwise-diff}
\end{equation}
 almost everywhere, where we interpret the right hand side to be $0$
whenever $f=0$. 

Indeed, this will follow from Proposition \ref{prop:pointwise-der}.
More precisely, let $K=\overline{\left\{ x:\ f(x)\ne0\right\} }$
denote the support of $f$. As a convex function $\phi$ is differentiable
almost everywhere on $K$, so at almost every $x\in K$ we may apply
Proposition \ref{prop:pointwise-der} and deduce that 
\[
\left.\frac{\dd}{\dd t}\right|_{t=0^{+}}\phi_{t}(x)=-\alpha\left(\nabla\phi(x)\right).
\]
 By the chain rule we then have 
\[
\left.\frac{\dd}{\dd t}\right|_{t=0^{+}}f_{t}(x)=\left.\frac{\dd}{\dd t}\right|_{t=0^{+}}e^{-\phi_{t}}(x)=\alpha\left(\nabla\phi(x)\right)\cdot f(x)
\]
 like we wanted. On the other hand, for every $x\notin K$ we have
$d(x,K)=\delta>0$, and then for every $t<\frac{\delta}{m}$ we have
$\widetilde{f}_{t}(x)=0$. Hence $f_{t}(x)=0$ as well for all small
enough $t>0$, so we obviously have $\lim_{t\to0^{+}}v_{t}(x)=0$.
This concludes the proof of \eqref{eq:pointwise-diff}. 

The same proof with $g$ replaced by $\widetilde{g}$ shows that $w_{t}\to\left(m\left|\nabla\phi\right|+c\right)f$
almost everywhere, and basic calculus implies that $u_{t}\to-mf$.
Let us call these three pointwise limits $v$, $w$ and $u$. It is
trivial that 
\[
\int u_{t}\xrightarrow{t\to0^{+}}-m\int f=\int u>-\infty.
\]
Since $f$ is essentially continuous we may apply Proposition \ref{prop:first-variation-ball}
and deduce that 
\[
\lim_{t\to0^{+}}\int w_{t}=\left.\frac{\dd}{\dd t}\right|_{t=0^{+}}\int\widetilde{f}_{t}=\int h_{\widetilde{g}}\dd S_{f}=\int w.
\]
Moreover, remark \ref{rem:finiteness-linear} explains why $\int w<\infty$. 

Hence we may apply Lemma \ref{lem:fatou-comparison} and deduce that
\[
\lim_{t\to0^{+}}\int v_{t}=\int\lim_{t\to0^{+}}v_{t}=\int\alpha\left(\nabla\phi\right)f=\int h_{g}\dd S_{f},
\]
completing the proof. 
\end{proof}
Now we can finally prove Theorem \ref{thm:first-variation} in its
full generality:
\begin{proof}[Proof of Theorem \ref{thm:first-variation}.]
First we claim that translating $g$ does not change the validity
of the theorem. Indeed, the left hand side clearly doesn't change
when we replace $g$ by $\widetilde{g}(x)=g(x-v)$. For the right
hand side we have $h_{\widetilde{g}}(x)=h_{g}(x)+\left\langle x,v\right\rangle $,
so by Proposition \pageref{prop:int-grad} we have 
\begin{align*}
\int h_{\widetilde{g}}\dd S_{f} & =\int h_{g}\dd S_{f}+\int\left\langle x,v\right\rangle \dd S_{f}=\int h_{g}\dd S_{f}+\int\left\langle \nabla\phi,v\right\rangle f\\
 & =\int h_{g}\dd S_{f}-\left\langle \int\nabla f,v\right\rangle =\int h_{g}\dd S_{f}.
\end{align*}
Hence we may translate and assume that $\max g=g(0)>0$, which means
that $\min\beta=\beta(0)<\infty$. This implies that $\alpha\ge-\beta(0)>-\infty$
is bounded from below.

For every integer $i>0$ we define 
\[
g_{i}(x)=\begin{cases}
g(x) & \left|x\right|\le i\\
0 & \text{otherwise.}
\end{cases}
\]
 Define $f_{t,i}=f\star t\cdot g_{i}$ and $\alpha_{i}=h_{g_{i}}$.
We claim that $\alpha_{i}(x)\nearrow\alpha(x)$ and $f_{t,i}(x)\nearrow f_{t}(x)$
as $i\to\infty$, where $x\in\RR^{n}$ and $t>0$ are fixed. 

Let us show that $f_{t,i}(x)\nearrow f_{t}(x)$. Since $g_{i}$ is
increasing in $i$ and $g_{i}\le g$ it follows that $f_{t,i}$ is
also increasing in $i$ and $f_{t,i}\le f$. Therefore we only have
to prove that 
\[
\sup_{i\ge1}f_{t,i}(x)\ge f_{t}(x).
\]
 Indeed, for every $\epsilon>0$ there exists $y_{0}\in\RR^{n}$ such
that 
\[
f_{t}(x)=\sup_{y\in\RR^{n}}f(x-y)g\left(\frac{y}{t}\right)^{t}\le f(x-y_{0})g\left(\frac{y_{0}}{t}\right)^{t}+\epsilon.
\]
 Therefore for every $i_{0}>\left|y_{0}/t\right|$ we have
\begin{align*}
\sup_{i\ge1}f_{t,i}(x) & \ge f_{t,i_{0}}(x)=\sup_{y\in\RR^{n}}f(x-y)g_{i_{0}}\left(\frac{y}{t}\right)^{t}\ge f(x-y_{0})g_{i_{0}}\left(\frac{y_{0}}{t}\right)^{t}\\
 & =f(x-y_{0})g\left(\frac{y_{0}}{t}\right)^{t}\ge f_{t}(x)-\epsilon.
\end{align*}
Since $\epsilon>0$ was arbitrary the claim is proved. The proof that
$\alpha_{i}(x)\to\alpha(x)$ is similar.

Note that 
\[
\alpha_{i}(y)=\sup_{x\in\RR^{n}}\left(\left\langle x,y\right\rangle -\beta_{i}(x)\right)=\sup_{\left|x\right|\le i}\left(\left\langle x,y\right\rangle -\beta(x)\right)\le i\left|y\right|-\beta(0),
\]
 and that $\alpha_{i}(y)\ge-\beta_{i}(0)=-\beta(0)$. Hence we may
apply Lemma \ref{lem:var-linear-growth} and conclude that
\[
\delta(f,g_{i})=\lim_{t\to0^{+}}\frac{\int f_{t,i}-\int f}{t}=\int h_{g_{i}}\dd S_{f}.
\]
 In particular, as was explained in Remark \ref{rem:finiteness-linear}
these expressions are finite. 

Since $\alpha_{i}\nearrow\alpha$ we have by monotone convergence
$\int h_{g_{i}}\dd S_{f}\nearrow\int h_{g}\dd S_{f}$. For the left
hand side, define $\rho_{i}(t)=\int f_{t,i}$ and $\rho(t)=\int f_{t}$,
and set $\rho_{i}(0)=\rho(0)=\int f$. By Proposition \ref{prop:lc-int-conv}
we have $\rho_{i}(t)\xrightarrow{i\to\infty}\rho(t)$ for all $t>0$. 

For every $i$ the function 
\[
\phi_{i,t}(x)=\left(\psi+t\alpha_{i}\right)^{\ast}(x)=\sup_{y\in\RR^{n}}\left[\left\langle x,y\right\rangle -\psi(y)-t\alpha_{i}(y)\right]
\]
 is jointly convex in $(t,x)\in\RR^{n+1}$ as the supremum of linear
functions. The Prékopa-Leindler inequality then implies that $\rho_{i}(t)=\int e^{-\phi_{i,t}(x)}\dd x$
is log-concave as well. Similarly $\rho$ is log-concave. Hence 
\begin{align*}
\left(\log\rho\right)_{+}^{\prime}(0) & =\lim_{t\to0^{+}}\frac{\log\rho(t)-\log\rho(0)}{t}=\sup_{t>0}\frac{\log\rho(t)-\log\rho(0)}{t}\\
 & =\sup_{t>0}\sup_{i}\frac{\log\rho_{i}(t)-\log\rho_{i}(0)}{t}=\sup_{i}\sup_{t>0}\frac{\log\rho_{i}(t)-\log\rho_{i}(0)}{t}\\
 & =\sup_{i}\left(\log\rho_{i}\right)_{+}^{\prime}(0)=\sup_{i}\frac{\left(\rho_{i}\right)_{+}^{\prime}(0)}{\int f}=\sup_{i}\frac{\int\alpha_{i}\dd S_{f}}{\int f}=\frac{\int\alpha\dd S_{f}}{\int f}.
\end{align*}
 On the other hand, we also have $\left(\log\rho\right)_{+}^{\prime}(0)=\frac{\rho_{+}^{\prime}(0)}{\int f}$.
One has to be careful here, since we do not know if $\rho$ is even
continuous at $t=0$, let alone differentiable. Therefore we interpret
this equality to mean that if $\left(\log\rho\right)_{+}^{\prime}(0)=+\infty$
then $\rho_{+}^{\prime}(0)=+\infty$ as well. Under this convention
we see that indeed 
\[
\lim_{t\to0^{+}}\frac{\int f_{t}-\int f}{t}=\rho_{+}^{\prime}(0)=\int\alpha\dd S_{f},
\]
 and the proof is complete. 
\end{proof}

\section{\label{sec:Lp-intro}Minkowski's theorem and $L^{p}$-surface area
measures}

It is now time to discuss a classic problem we avoided so far: What
measures are surface area measures? In the classic case of convex
bodies, the answer is known as Minkowski's existence theorem:
\begin{thm}
\label{thm:minkowski-classic}Let $\mu$ be a finite Borel measure
on $\SS^{n-1}$. Then $\mu=S_{K}$ for some convex body $K$ if and
only if it satisfies the following two conditions:
\begin{enumerate}
\item $\mu$ is centered, i.e. $\int_{\SS^{n-1}}x\dd\mu(x)=0$.
\item $\mu$ is not supported on any great sub-sphere of $\SS^{n-1}$.
\end{enumerate}
\end{thm}

In this classical setting, the uniqueness is also well known \textendash{}
if $S_{K}=S_{L}$ for some convex bodies $K$ and $L$ then necessarily
$K=L+v$ for some $v\in\RR^{n}$. For a proof of these facts see e.g.
Sections 8.1 and 8.2 of \cite{Schneider2013}.

In \cite{Cordero-Erausquin2015} Cordero-Erausquin and Klartag proved
a functional version of Theorem \ref{thm:minkowski-classic}. In our
notation their result reads as follows:
\begin{thm}[\cite{Cordero-Erausquin2015}]
\label{thm:1-lc-minkowski}Let $\mu$ be a finite Borel measure on
$\RR^{n}$. Then $\mu=S_{f}$ for an essentially continuous $f\in\lc_{n}$
with $0<\int f<\infty$ if and only $\mu$ satisfies the following
two conditions:
\begin{enumerate}
\item $\mu$ is centered (so in particular $\int_{\RR^{n}}\left|x\right|\dd\mu(x)<\infty$).
\item $\mu$ is not supported on any lower dimensional linear subspace of
$\RR^{n}$.
\end{enumerate}
Moreover, $f$ is uniquely determined up to translations.
\end{thm}

Besides its geometric content this theorem can also be viewed analytically,
as an existence and uniqueness theorem for generalized solutions of
a Monge-Ampère type differential equation. Indeed, assume that $\frac{\dd\mu}{\dd x}=g$
for a smooth function $g$, and that the solution $f=e^{-\phi}$ to
the equation $S_{f}=\mu$ is smooth as well. Then using \eqref{eq:push-forward}
and the classic change of variables formula we see that 
\[
g\left(\nabla\phi(x)\right)\cdot\det\left(\nabla^{2}\phi(x)\right)=e^{-\phi(x)}
\]
 for all $x\in\RR^{n}$. This point of view is further explained in
\cite{Cordero-Erausquin2015}. We also remark that the ``uniqueness''
part of the theorem was also proved in \cite{Colesanti2013} under
some technical conditions. 

For the rest of this paper we will be mostly interested in an extension
of the surface area measure, known as the $L^{p}$-surface area measure.
Given two convex bodies $K,L$ containing the origin and a number
$p\ge1$, the $L^{p}$-combination $K+_{p}\left(t\cdot L\right)$
of $K$ and $L$ is defined via its support function by 
\begin{equation}
h_{K+_{p}t\cdot L}=\left(h_{K}^{p}+t\cdot h_{L}^{p}\right)^{1/p}.\label{eq:p-mean-direct}
\end{equation}
 Note that since $0\in K,L$ we know that $h_{K},h_{L}\ge0$, so the
right hand side is well-defined. Moreover, using the fact that the
$L^{p}$ norm is indeed a norm it is not hard to check that the right
hand side defines a convex, $1$-homogeneous function, and hence the
body $K+_{p}t\cdot L$ exists. For $p=1$ the $L^{p}$-addition $+_{1}$
coincides with the usual Minkowski addition. $L^{p}$ additions of
convex bodies were first defined by Firey (\cite{Firey1962}), and
the Brunn-Minkowski theory of such bodies was developed by Lutwak
(\cite{Lutwak1993}, \cite{Lutwak1996}). In particular, in \cite{Lutwak1993}
Lutwak proved an extension of \eqref{eq:body-area-measure} for this
case: For every pair convex bodies $K,L$ containing the origin we
have 
\begin{equation}
\lim_{t\to0^{+}}\frac{\left|K+_{p}t\cdot L\right|-\left|K\right|}{t}=\frac{1}{p}\int_{\SS^{n-1}}h_{L}^{p}h_{K}^{1-p}\dd S_{K}.\label{eq:p-variation}
\end{equation}

We refer to the measure $h_{K}^{1-p}\dd S_{K}$ as the $L^{p}$-surface
areas measure of $K$, and denote it by $S_{K,p}$. In the same paper
Lutwak proved an extension of Minkowski's existence theorem for $p\ge1$:
\begin{thm}[\cite{Lutwak1993}]
\label{thm:p-Minkowski-classic}Fix $p\ge1$, and let $\mu$ be an
\textbf{even} finite Borel measure on $\SS^{n-1}$ which is not supported
on any great sub-sphere. Then:
\begin{enumerate}
\item If $p\ne n$ there exists an origin-symmetric convex body $K$ (i.e.
$K=-K$) such that $S_{K,p}=\mu$. 
\item For $p=n$ there exists an origin-symmetric convex body $K$ such
that $S_{K,n}=c\cdot\mu$ for some $c>0$. 
\end{enumerate}
Moreover, the body $K$ is unique. 
\end{thm}

In order to explain some peculiarities about the statement of Theorem
\ref{thm:p-Minkowski-classic}, it is useful to say a few words about
its proof. In the proof one finds the body $K$ by minimizing the
functional 
\[
\Phi(L)=\left|L\right|^{-\frac{p}{n}}\int_{\SS^{n-1}}h_{L}^{p}\dd\mu
\]
 over the class of origin symmetric convex bodies. If the minimum
is attained at some body $K$, then it turns out that the first order
optimality condition ``$\nabla\Phi(K)=0$'' implies that $S_{K}=c\cdot\mu$
for some $c>0$. If $p\ne n$ we can dilate $K$ to have exactly $S_{K}=\mu$,
by noticing that $S_{\lambda K,p}=\lambda^{n-p}S_{K,p}$. Obviously
this scaling idea cannot work when $p=n$, as in this case $S_{\lambda K,n}=S_{K,n}$.
This explains why the case $p=n$ is special.

This very rough sketch of the proof can also explain why $\mu$ is
assumed to be even, an assumption that was unnecessary in the case
$p=1$. Indeed, the use of the first order optimality condition ``$\nabla\Phi(K)=0$''
requires the minimizer $K$ to be an interior point of the domain
of $\Phi$. When $p>1$, $\Phi$ can only be defined on convex bodies
containing the origin, to make $h_{L}^{p}$ well defined. Without
the assumption that $L$ is origin symmetric the minimum of $\Phi$
may be obtained at some body $K$ containing $0$ at its boundary,
which will make the argument impossible. Variants of Theorem \ref{thm:p-Minkowski-classic}
are known for non-even measures (see \cite{Chou2006} and \cite{Hug2005}),
but we will not require them here. 

In recent years there has been a lot of interest in the $L^{p}$ theory,
and in particular in $L^{p}$-surface area measures, for $0<p<1$.
When $p<1$ one cannot define $K+_{p}t\cdot L$ using \eqref{eq:p-mean-direct},
as the right hand side is not necessarily a convex function. Instead,
for any function $\rho:\SS^{n-1}\to(0,\infty)$, convex or not, one
defines the Alexandrov body (or Wulff shape) of $\rho$ as 
\[
A\left[\rho\right]=\left\{ x\in\RR^{n}:\ \left\langle x,\theta\right\rangle \le\rho(\theta)\text{ for all }\theta\in\SS^{n-1}\right\} .
\]
 In other words, $A[\rho]$ is the largest convex body with $h_{A[\rho]}\le\rho$.
In particular for every convex body $K$ we have $A[h_{K}]=K$. Then
one can define 
\[
K+_{p}t\cdot L=A\left[\left(h_{K}^{p}+th_{L}^{p}\right)^{1/p}\right]
\]
 for any $p>0$. Using the saw called Alexandrov Lemma, one can verify
that \eqref{eq:p-variation} remains true for $0<p<1$. Furthermore,
as Schneider observes in \cite{Schneider2013} (see Theorem 9.2.1),
the existence part of Theorem \ref{thm:p-Minkowski-classic} continues
to hold in this case, with the same proof. However, for $p<1$, the
uniqueness problem is highly non-trivial. In fact it was proved by
Böröczky, Lutwak, Yang and Zhang (\cite{Boroczky2012}) that this
uniqueness problem is equivalent to the so called $L^{p}$-Brunn-Minkowski
conjecture, a major open problem in convex geometry. While this relation
was one of our original motivation to study $L^{p}$-surface area
measures it will not play any role in the sequel, so we will not give
any further details. 

We will be interested in functional $L^{p}$-addition and functional
$L^{p}$-surface area measures. The definitions are fairly straightforward:
\begin{defn}
Let $\psi:\RR^{n}\to(-\infty,\infty]$ be a lower semi-continuous
function (which may or may not be convex). The Alexandrov Function
of $\psi$ is $f=A\left[\psi\right]=e^{-\psi^{\ast}}.$ 
\end{defn}

Note that $h_{f}=\psi^{\ast\ast}$ , so $f$ is the largest log-concave
function with $h_{f}\le\psi$ in analogy to the classical theory.
We then define: 
\begin{defn}
Fix $p>0$ and fix functions $f,g\in\lc_{n}$ with $h_{f},h_{g}\ge0$.
The $L^{p}$-combination $f\star_{p}t\cdot g$ is defined by
\[
f\star_{p}t\cdot g=A\left[\left(h_{f}^{p}+th_{g}^{p}\right)^{1/p}\right].
\]
\end{defn}

Let us compute the first variation of $\int\left(f\star_{p}t\cdot g\right)$.
Unlike the case $p=1$ we will not do it rigorously under minimal
assumptions, but use \eqref{eq:legendre-variation-general} to derive
the answer formally assuming sufficient regularity:
\begin{align*}
\left.\frac{\dd}{\dd t}\right|_{t=0^{+}}\int\left(f\star_{p}t\cdot g\right) & =\int\left.\frac{\dd}{\dd t}\right|_{t=0^{+}}\left(f\star_{p}t\cdot g\right)\\
 & =\int\left.\frac{\dd}{\dd t}\right|_{t=0^{+}}\exp\left(-\left[\left(h_{f}^{p}+th_{g}^{p}\right)^{1/p}\right]^{\ast}\right)\\
 & =-\int\left(e^{-\phi}\cdot\left.\frac{\dd}{\dd t}\right|_{t=0^{+}}\left[\left(h_{f}^{p}+th_{g}^{p}\right)^{1/p}\right]^{\ast}\right)\\
 & =\int\left(e^{-\phi}\cdot\left(\left.\frac{\dd}{\dd t}\right|_{t=0^{+}}\left[\left(h_{f}^{p}+th_{g}^{p}\right)^{1/p}\right]\circ\nabla\phi\right)\right)\\
 & =\frac{1}{p}\int e^{-\phi}\cdot h_{f}^{1-p}\left(\nabla\phi\right)\cdot h_{g}^{p}\left(\nabla\phi\right)=\frac{1}{p}\int h_{g}^{p}h_{f}^{1-p}\dd S_{f}.
\end{align*}
 As expected, the result is completely analogous to the case of convex
bodies, so we make the following definition:
\begin{defn}
For $f\in\lc_{n}$ with $0<\int f<\infty$ and $0<p<1$ we define
the $L^{p}$-surface area measure of $f$ to be $S_{f,p}=h_{f}^{1-p}\dd S_{f}$. 
\end{defn}

\begin{rem}
\label{rem:surface-support}For $S_{f,p}$ to be well defined and
not identically equal to $+\infty$ we should verify that $h_{f}(x)<\infty$
for $S_{f}$-almost every $x$. This is true since at every point
$x\in\RR^{n}$ where $f(x)>0$ and $\phi(x)=-\log f(x)$ is differentiable
we have by Lemma \ref{lem:unique-maximizer} 
\[
\phi(x)=\left\langle x,\nabla\phi(x)\right\rangle -h_{f}\left(\nabla\phi(x)\right),
\]
 so in particular $h_{f}\left(\nabla\phi(x)\right)<\infty$. Hence
\[
S_{f}\left(\left\{ y:\ h_{f}(y)=\infty\right\} \right)=\int\oo_{\left\{ h_{f}\left(\nabla\phi(x)\right)=\infty\right\} }f(x)\dd x=0,
\]
 so $S_{f}$-almost everywhere we have $h_{f}<\infty$. 
\end{rem}

The second major goal of this paper is to prove a Minkowski existence
theorem for functional $L^{p}$-surface area measures. More concretely
we will prove the following:
\begin{thm}
\label{thm:lp-minkowski}Fix $0<p<1$. Let $\mu$ be an \textbf{even}
finite Borel measure on $\RR^{n}$. Assume that:
\begin{enumerate}
\item $\int\left|x\right|\dd\mu<\infty$ (and then of course $\mu$ is centered,
as it is even)
\item $\mu$ is not supported on any hyperplane
\end{enumerate}
Then there exists $c>0$ and an even function $f\in\lc_{n}$ with
$h_{f}\ge0$ such that $S_{f,p}=c\cdot\mu$.
\end{thm}

If again we assume that $\frac{\dd\mu}{\dd x}=g$ for some smooth
function $g$ and that the solution $f=e^{-\phi}$ to $S_{f,p}=c\cdot\mu$
is also smooth, then $\phi$ solves the Monge-Ampère type differential
equation

\[
c\cdot g\left(\nabla\phi(x)\right)\cdot\det\left(\nabla^{2}\phi(x)\right)=\left(\phi^{\ast}(x)\right)^{1-p}e^{-\phi(x)}.
\]

Note that we claim nothing about the uniqueness of $f$. As was explained
above this is a much more delicate issue that we will not address
here. Also note that we only prove the result for even measures $\mu$,
and we can only deduce that $S_{f,p}$ coincides with the measure
$\mu$ up to a constant $c>0$. This is very similar to the case $p=n$
of Theorem \ref{thm:p-Minkowski-classic}, and happens for essentially
the same reasons.

We will prove Theorem \ref{thm:lp-minkowski} in the next section.

\section{\label{sec:Lp-existence}A functional $L^{p}$ Minkowski's existence
theorem}

In this section we prove Theorem \ref{thm:lp-minkowski}. Not surprisingly,
we will find the function $f$ we are looking for by solving a certain
optimization problem. Unlike the proofs of Theorems \ref{thm:1-lc-minkowski}
and \ref{thm:p-Minkowski-classic} however, it will be much more convenient
to work with a \emph{constrained} optimization problem. First we will
need a result guaranteeing the existence of a minimizer:
\begin{prop}
\label{prop:lp-minimizer}Assume $\mu$ satisfies the assumptions
of Theorem \ref{thm:lp-minkowski}. Fix $0<p<1$, and consider the
minimization problem 
\[
\min\left\{ \int\psi^{p}\dd\mu:\ \begin{array}{l}
\psi:\RR^{n}\to[0,\infty]\text{ is even,}\\
\text{measurable and }\int e^{-\psi^{\ast}}\ge a
\end{array}\right\} .
\]
If $a>0$ is large enough then the minimum is attained for a lower
semi-continuous convex function $\psi_{0}$. Moreover, $\psi_{0}(0)>0$
and $\int e^{-\psi_{0}^{\ast}}=a$. 
\end{prop}

In order to prove this proposition we will need two lemmas, which
are both variants of lemmas from \cite{Cordero-Erausquin2015}. First
let us state our version of Lemma 15 from this paper:
\begin{lem}
\label{lem:zero-value}Assume $\mu$ satisfies the assumptions of
Theorem \ref{thm:lp-minkowski}. Let $\psi:\RR^{n}\to[0,\infty]$
be an even lower semi-continuous convex function with $\psi(0)=0$.
Write $\phi=\psi^{\ast}$ and fix $0<p<1$. Then 
\[
\int\psi^{p}\dd\mu\ge c_{\mu}\left(\int e^{-\phi}\right)^{\frac{p}{n}}-C_{\mu}
\]
 for $c_{\mu},C_{\mu}>0$ that depend on $\mu$ and $p$ but not on
$\psi$.
\end{lem}

For completeness we provide a proof of the lemma. The fact that we
only deal with even functions and measures makes the proof shorter
than the corresponding proof in \cite{Cordero-Erausquin2015}:
\begin{proof}
If $\int\psi^{p}\dd\mu=\infty$ there is nothing to prove, so we may
assume that $\int\psi^{p}\dd\mu<\infty$. Therefore $\psi$ is finite
on the support of $\mu$, and since $\psi$ is convex it is also finite
on its convex hull $K=\text{conv}(\text{supp}(\mu))$. By our assumptions
on $\mu$ the body $K$ is an origin symmetric convex body with non-empty
interior, so it must contain $0$ in its interior. In particular $\psi$
is bounded in a neighborhood of $0$. If $\left|\psi(y)\right|\le M$
for $\left|y\right|\le\delta$ then 
\[
\phi(x)=\sup_{y\in\RR^{n}}\left[\left\langle x,y\right\rangle -\psi(y)\right]\ge\left\langle x,\frac{\delta x}{\left|x\right|}\right\rangle -\psi\left(\frac{\delta x}{\left|x\right|}\right)\ge\delta\left|x\right|-M,
\]
 So in particular $\int e^{-\phi}<\infty$. If $\int e^{-\phi}=0$
again there is nothing to prove, so we may assume $0<\int e^{-\phi}<\infty$.
By the Blaschke-Santaló inequality (see \cite{Artstein-Avidan2004})
it follows that 
\[
\int e^{-\psi}\cdot\int e^{-\phi}\le\left(2\pi\right)^{\frac{n}{2}}.
\]

Next we define 
\[
K=\left\{ y\in\RR^{n}:\ \psi(y)\le1\right\} .
\]
 We also define $r=\min_{\theta\in\SS^{n-1}}h_{K}(\theta)$, and let
$\theta_{0}\in\SS^{n-1}$ be the direction in which this minimum is
attained. It follows that
\[
\frac{\left(2\pi\right)^{n/2}}{\int e^{-\phi}}\ge\int_{\RR^{n}}e^{-\psi}\ge\int_{K}e^{-\psi}\ge\frac{1}{e}\left|K\right|\ge\frac{1}{e}\left|rB^{n}\right|=c_{n}\cdot r^{n},
\]
 so $r\le C_{n}\cdot\left(\int e^{-\phi}\right)^{-1/n}$. Here and
everywhere else $c_{n},C_{n}>0$ are some constants that depend only
on the dimension $n$. 

For every $y\in K$ we have $r=h_{K}(\theta_{0})\ge\left|\left\langle y,\theta_{0}\right\rangle \right|$.
Equivalently, if $\left|\left\langle y,\theta_{0}\right\rangle \right|>r$
then $x\notin K$, so $\psi(y)>1$. It follows that if $2r\le\left|\left\langle y,\theta_{0}\right\rangle \right|$
then 
\[
1<\psi\left(\frac{2r}{\left|\left\langle y,\theta_{0}\right\rangle \right|}y\right)=\psi\left(\left(1-\frac{2r}{\left|\left\langle y,\theta_{0}\right\rangle \right|}\right)\cdot0+\frac{2r}{\left|\left\langle y,\theta_{0}\right\rangle \right|}\cdot y\right)\le\frac{2r}{\left|\left\langle y,\theta_{0}\right\rangle \right|}\psi(y),
\]
 so $\psi(y)+1\ge\psi(y)\ge\frac{\left|\left\langle y,\theta_{0}\right\rangle \right|}{2r}$.
Obviously if $\left|\left\langle y,\theta_{0}\right\rangle \right|<2r$
this inequality still holds trivially, so it holds for every $y\in\RR^{n}$.
Hence 
\[
\int(\psi^{p}+1)\dd\mu\ge\int(\psi+1)^{p}\dd\mu\ge\frac{1}{\left(2r\right)^{p}}\int\left|\left\langle y,\theta_{0}\right\rangle \right|^{p}\dd\mu.
\]
 The function $\theta\mapsto\int\left|\left\langle y,\theta\right\rangle \right|^{p}\dd\mu(y)$
is continuous on $\SS^{n-1}$ by the dominated convergence theorem,
and is strictly positive since $\mu$ is not supported on any hyperplane.
Hence it has a positive minimum which we may denote by $\widetilde{c}_{\mu}$.
It follows that 
\[
\int(\psi^{p}+1)\dd\mu\ge\frac{\widetilde{c}_{\mu}}{(2r)^{p}}\ge c_{\mu}\left(\int e^{-\phi}\right)^{\frac{p}{n}},
\]
 completing the proof.
\end{proof}
The second lemma we will need is a variant of Lemma 17 from \cite{Cordero-Erausquin2015}:
\begin{lem}
\label{lem:p-compactness}Assume $\mu$ satisfies the assumptions
of Theorem \ref{thm:lp-minkowski}. Let $\left\{ \psi_{i}\right\} _{i=1}^{\infty}$
be non-negative, even, lower semi-continuous convex functions such
that 
\[
\sup_{i}\int\psi_{i}^{p}\dd\mu<\infty.
\]
 There there exists a subsequence $\left\{ \psi_{i_{j}}\right\} _{j=1}^{\infty}$
and an even lower semi-continuous convex function $\psi:\RR^{n}\to[0,\infty]$
such that 
\[
\int\psi^{p}\dd\mu\le\liminf_{j\to\infty}\int\psi_{i_{j}}^{p}\dd\mu\quad\text{and}\quad\int e^{-\psi^{\ast}}\ge\limsup_{j\to\infty}\int e^{-\psi_{i_{j}}^{\ast}}.
\]
 
\end{lem}

In fact, Lemma 17 of \cite{Cordero-Erausquin2015} is exactly the
same statement in the case $p=1$, and with the assumption of evenness
replaced with the assumption $\psi_{i}(0)=0$. The proofs are identical,
as writing the extra power $p$ everywhere doesn't affect the argument
in any way. The assumption $\psi_{i}(0)=0$ is only used in \cite{Cordero-Erausquin2015}
to know that $\psi_{i}(\lambda x)$ is increasing in $\lambda$. This
is trivial when $\psi_{i}$ is even, so this assumption may be omitted.
Since the proofs are otherwise identical we omit a proof of Lemma
\ref{lem:p-compactness}. 

Using these two lemma we can prove Proposition \ref{prop:lp-minimizer}:
\begin{proof}[Proof of Proposition \ref{prop:lp-minimizer}.]
Without loss of generality assume that $\mu$ is a probability measure.
Consider the function $\widetilde{\psi}(y)=\log a+c_{n}\left|y\right|$,
where the constant $c_{n}$ is chosen such that $\left|c_{n}B_{2}^{n}\right|=1$.
Then $\widetilde{\psi}^{\ast}(x)=-\log a+\oo_{c_{n}B_{2}^{n}}^{\infty}$,
so $\int e^{-\widetilde{\psi}^{\ast}}=a$. For $a\ge e$ we have $\widetilde{\psi}\ge1$,
so 
\[
\int\widetilde{\psi}^{p}\dd\mu\le\int\widetilde{\psi}\dd\mu=\log a+c_{n}\int\left|y\right|\dd\mu(y)=\log a+C_{\mu}.
\]
 In particular we see that 
\[
m=\inf\left\{ \int\psi^{p}\dd\mu:\ \begin{array}{l}
\psi:\RR^{n}\to[0,\infty]\text{ is even,}\\
\text{measurable and }\int e^{-\psi^{\ast}}\ge a
\end{array}\right\} <\infty.
\]

Next we choose a sequence $\left\{ \psi_{i}\right\} _{i=1}^{\infty}$
of even, measurable functions with $\int e^{-\psi_{i}^{\ast}}\ge a$
and such that $\int\psi_{i}^{p}\dd\mu\to m$. By replacing each $\psi_{i}$
with its second Legendre conjugate $\psi_{i}^{\ast\ast}$ we may assume
the functions $\left\{ \psi_{i}\right\} $ are all lower semi-continuous
and convex. Obviously $\int\psi_{i}^{p}\dd\mu<m+1$ for all but finitely
many values of $i$. Hence we can apply Lemma \ref{lem:p-compactness}
and find a subsequence $\left\{ \psi_{i_{j}}\right\} $ and an even
lower semi-continuous convex function $\psi:\RR^{n}\to[0,\infty]$
such that 
\[
\int\psi^{p}\dd\mu\le\liminf_{j\to\infty}\int\psi_{i_{j}}^{p}\dd\mu=m
\]
 and 
\[
\int e^{-\psi^{\ast}}\ge\limsup_{j\to\infty}\int e^{-\psi_{i_{j}}^{\ast}}\ge a.
\]
 It follows that $\int\psi^{p}\dd\mu=m$ and therefore $\psi$ is
the minimizer we were looking for. 

Assume by contradiction that $\psi(0)=0$. Then by Lemma \ref{lem:zero-value}
we have 
\[
\int\psi^{p}\dd\mu\ge c_{\mu}\left(\int e^{-\psi^{\ast}}\right)^{\frac{p}{n}}-C_{\mu}\ge c_{\mu}a^{\frac{p}{n}}-C_{\mu}.
\]
Therefore for large enough $a>0$ we have $\int\psi^{p}\dd\mu>\int\widetilde{\psi}^{p}\dd\mu$,
which is a contradiction to the minimality of $\psi$. Hence $\psi(0)>0$
for $a>0$ large enough. 

Finally, assume by contradiction that $\int e^{-\psi^{\ast}}>a$.
Since $\psi(0)>0$ the function $\psi_{\epsilon}=\psi-\epsilon$ is
non-negative for small enough $\epsilon>0$. Since $\int e^{-\psi_{\epsilon}^{\ast}}=e^{-\epsilon}\int e^{-\psi^{\ast}}$
we see that $\psi_{\epsilon}$ is in our domain for small enough $\epsilon>0$.
But this is impossible since $\psi$ is a minimizer and $\int\psi_{\epsilon}^{p}\dd\mu<\int\psi^{p}\dd\mu$.
This complete the proof. 
\end{proof}
Now that we have our minimizer, Theorem \ref{thm:lp-minkowski} will
follow by writing the first order optimality condition. In order to
do this, we will need the following result, which is a much simpler
variant of Theorem \ref{thm:first-variation}:
\begin{prop}
\label{prop:bounded-variation}Fix $f\in\lc_{n}$ with $0<\int f<\infty$
and set $\psi=h_{f}$. Let $v:\RR^{n}\to\RR$ be bounded and continuous.
Then 

\[
\left.\frac{\dd}{\dd t}\right|_{t=0}\int e^{-\left(\psi+tv\right)^{\ast}}=\int v\dd S_{f}.
\]
\end{prop}

\begin{proof}
Fix a point $x_{0}$ where $\phi=-\log f$ is differentiable. By Proposition
\ref{prop:pointwise-der} we know that 

\[
\left.\frac{\dd}{\dd t}\right|_{t=0^{+}}\left(\psi+tv\right)^{\ast}(x_{0})=-v\left(\nabla\phi(x_{0})\right).
\]
(Recall that we very explicitly did not assume in Proposition \ref{prop:pointwise-der}
that the function $v$ is convex). Applying the same proposition to
$-v$ instead of $v$ we see that 
\[
\left.\frac{\dd}{\dd t}\right|_{t=0^{-}}\left(\psi+tv\right)^{\ast}(x_{0})=-\left(\left.\frac{\dd}{\dd t}\right|_{t=0^{+}}\left(\psi-tv\right)^{\ast}(x_{0})\right)=-v\left(\nabla\phi(x_{0})\right),
\]
 so the two sided derivative exists. Therefore if we write $f_{t}=e^{-\left(\psi+tv\right)^{\ast}}$
then by the chain rule we have $\left.\frac{\dd}{\dd t}\right|_{t=0}f_{t}(x_{0})=v\left(\nabla\phi(x_{0})\right)f$.

Choose $M>0$ such that $\left|v\right|\le M$. Then the functions
$f_{t}=e^{-\left(\psi+tv\right)^{\ast}}$ satisfy $e^{-tM}f\le f_{t}\le e^{tM}f$.
In particular all functions $f_{t}$ have the same support which we
denote by $K$. Moreover for $\left|t\right|\le1$ we have
\begin{align*}
\left|\frac{f_{t}-f}{t}\right| & \le\max\left\{ \left|\frac{e^{tM}f-f}{t}\right|,\left|\frac{e^{-tM}f-f}{t}\right|\right\} \\
 & =f\cdot\max\left\{ \left|\frac{e^{tM}-1}{t}\right|,\left|\frac{e^{-tM}-1}{t}\right|\right\} \le\left(e^{M}-1\right)\cdot f
\end{align*}
 which is an integrable function. Hence by dominated convergence we
have 
\begin{align*}
\left.\frac{\dd}{\dd t}\right|_{t=0}\int e^{-\left(\psi+tv\right)^{\ast}} & =\left.\frac{\dd}{\dd t}\right|_{t=0}\int_{K}f_{t}=\int_{K}\left(\left.\frac{\dd}{\dd t}\right|_{t=0}f_{t}\right)=\int_{K}v\left(\nabla\phi\right)f\\
 & =\int v\left(\nabla\phi\right)f=\int v\dd S_{f}.
\end{align*}
 
\end{proof}
We can now complete the proof of Theorem \ref{thm:lp-minkowski}.
In theory, since we are working with a constrained optimization problem,
the first order optimality condition should involve Lagrange multipliers.
Luckily our functions are simple enough that we may compute this optimality
condition directly and we do not have to worry about the theory of
Lagrange multipliers on an infinite dimensional space: 
\begin{proof}[Proof of Theorem \ref{thm:lp-minkowski}]
Fix $a>0$ large enough and Let $\psi$ be the minimizer from Proposition
\ref{prop:lp-minimizer}. Write $f=e^{-\phi}=e^{-\psi^{\ast}}$ and
define $K=\overline{\left\{ x:\ \psi(x)<\infty\right\} }.$ Since
$\int\psi^{p}\dd\mu<\infty$ the measure $\mu$ is supported on $K$.
From Remark \ref{rem:surface-support} the measure $S_{f}$ is also
supported on $K$. 

Fix an even, bounded and continuous function $v:\RR^{n}\to\RR$ whose
support is contained in the \emph{interior} of $K$. Define 
\[
\psi_{t}=\psi+tv\psi^{1-p}+\log a-\log\int e^{-\left(\psi+tv\psi^{1-p}\right)^{\ast}}.
\]
 Note that 
\[
\psi_{t}^{\ast}=\left(\psi+tv\psi^{1-p}\right)^{\ast}-\log a+\log\int e^{-\left(\psi+tv\psi^{1-p}\right)^{\ast}}
\]
 so that $\int e^{-\psi_{t}^{\ast}}=a$. Moreover $v\psi^{1-p}$ is
a bounded function, since $v$ is bounded on $\RR^{n}$ and $\psi$
is bounded on the support of $v$. Let us write $\left|v\psi^{1-p}\right|\le M$.
We then also have 
\[
\log\int e^{-\left(\psi+tv\psi^{1-p}\right)^{\ast}}\le\log\int e^{-\left(\psi+\left|t\right|M\right)^{\ast}}=\log\left(e^{\left|t\right|M}\int e^{-\psi^{\ast}}\right)=\log a+\left|t\right|M
\]
 and similarly $\log\int e^{-\left(\psi+tv\psi^{1-p}\right)^{\ast}}\ge\log a-\left|t\right|M$.
Hence $\left|\psi_{t}-\psi\right|\le2\left|t\right|M$. In particular,
since $\min\psi=\psi(0)>0$, there exists $\delta>0$ such that $\psi_{t}\ge\delta$
if $\left|t\right|$ is small enough.

By Proposition \ref{prop:bounded-variation} we have 
\[
\left.\frac{\dd}{\dd t}\right|_{t=0}\psi_{t}=v\psi^{1-p}-\frac{1}{a}\cdot\int v\psi^{1-p}\dd S_{f}.
\]
 The function $x\mapsto x^{p}$ is $\frac{p}{\delta^{1-p}}$-Lipschitz
on the interval $[\delta,\infty)$. Since for small enough $\left|t\right|$
we have $\psi,\psi_{t}\ge\delta$ we obtain
\begin{equation}
\left|\frac{\psi_{t}^{p}-\psi^{p}}{t}\right|\le\frac{p}{\delta^{1-p}}\left|\frac{\psi_{t}-\psi}{t}\right|\le\frac{2Mp}{\delta^{1-p}}.\label{eq:p-dom-bound}
\end{equation}

From the fact that $\psi$ is a minimizer it follows that $\int\psi_{t}^{p}\dd\mu\ge\int\psi^{p}\dd\mu$
for all $\left|t\right|$ small enough. Because of \eqref{eq:p-dom-bound}
we may apply dominated convergence and conclude that
\begin{align*}
0 & =\left.\frac{\dd}{\dd t}\right|_{t=0}\int\psi_{t}^{p}\dd\mu=\int\left(\left.\frac{\dd}{\dd t}\right|_{t=0}\psi_{t}^{p}\right)\dd\mu\\
 & =\int\left(p\psi^{p-1}\cdot\left(v\psi^{1-p}-\frac{1}{a}\int v\psi^{1-p}\dd S_{f}\right)\right)\dd\mu\\
 & =p\int v\dd\mu-\frac{1}{a}\int v\psi^{1-p}\dd S_{f}\cdot\int p\psi^{p-1}\dd\mu.
\end{align*}
We see that 
\[
\int v\dd\mu=c\cdot\int v\psi^{1-p}\dd S_{f}=c\cdot\int v\dd S_{f,p}
\]
 for some constant $c>0$ that depends on $\mu$ and $\psi$ but not
on $v$. Since $\mu$ and $S_{f,p}$ are even and supported on $K$,
and since this equality holds for \emph{every }even, bounded and continuous
function $v$ whose support is contained in the interior of $K$,
it follows that $\mu=c\cdot S_{f,p}$ . This completes the proof.
\end{proof}
\bibliographystyle{plain}
\bibliography{paper.bbl}

\end{document}